\documentclass[10pt,reqno]{amsart}
\usepackage{latexsym,amsmath,amssymb,amscd}
\usepackage[all]{xy}
\usepackage{enumerate}

\def\today{\ifcase \month \or
   January \or February \or March \or April \or
   May \or June \or July \or August \or
   September \or October \or November \or December \fi
   \space\number\day , \number\year}
   
\makeatletter
  \newcommand\@dotsep{4.5}
  \def\@tocline#1#2#3#4#5#6#7{\relax
     \ifnum #1>\c@tocdepth 
     \else
     \par \addpenalty\@secpenalty\addvspace{#2}%
     \begingroup \hyphenpenalty\@M
     \@ifempty{#4}{%
     \@tempdima\csname r@tocindent\number#1\endcsname\relax
        }{%
         \@tempdima#4\relax
           }%
      \parindent\z@ \leftskip#3\relax \advance\leftskip\@tempdima\relax
      \rightskip\@pnumwidth plus1em \parfillskip-\@pnumwidth
       #5\leavevmode\hskip-\@tempdima #6\relax
       \leaders\hbox{$\m@th
       \mkern \@dotsep mu\hbox{.}\mkern \@dotsep mu$}\hfill
       \hbox to\@pnumwidth{\@tocpagenum{#7}}\par
       \nobreak
        \endgroup
         \fi}
\makeatother 

\begin{document}


\makeatletter
\@addtoreset{figure}{section}
\def\thefigure{\thesection.\@arabic\c@figure}
\def\fps@figure{h,t}
\@addtoreset{table}{bsection}

\def\thetable{\thesection.\@arabic\c@table}
\def\fps@table{h, t}
\@addtoreset{equation}{section}
\def\theequation{
\arabic{equation}}
\makeatother

\newcommand{\bfi}{\bfseries\itshape}

\newtheorem{theorem}{Theorem}
\newtheorem{acknowledgment}[theorem]{Acknowledgment}
\newtheorem{corollary}[theorem]{Corollary}
\newtheorem{definition}[theorem]{Definition}
\newtheorem{example}[theorem]{Example}
\newtheorem{lemma}[theorem]{Lemma}
\newtheorem{notation}[theorem]{Notation}
\newtheorem{problem}[theorem]{Problem}
\newtheorem{proposition}[theorem]{Proposition}
\newtheorem{question}[theorem]{Question}
\newtheorem{remark}[theorem]{Remark}
\newtheorem{setting}[theorem]{Setting}

\numberwithin{theorem}{section}
\numberwithin{equation}{section}

\renewcommand{\1}{{\bf 1}}
\newcommand{\Ad}{{\rm Ad}}
\newcommand{\Aut}{{\rm Aut}\,}
\newcommand{\ad}{{\rm ad}}
\newcommand{\botimes}{\bar{\otimes}}
\newcommand{\Ci}{{\mathcal C}^\infty}
\newcommand{\Cl}{{\rm Cl}\,}
\newcommand{\de}{{\rm d}}
\newcommand{\dr}{{\rm dr}\,}
\newcommand{\ee}{{\rm e}}
\newcommand{\End}{{\rm End}\,}
\newcommand{\id}{{\rm id}}
\newcommand{\ie}{{\rm i}}
\newcommand{\Index}{{\rm Index}}
\newcommand{\GL}{{\rm GL}}
\newcommand{\Gr}{{\rm Gr}}
\newcommand{\Hom}{{\rm Hom}\,}
\newcommand{\Ind}{{\rm Ind}}
\newcommand{\Ker}{{\rm Ker}\,}
\newcommand{\pr}{{\rm pr}}
\newcommand{\Ran}{{\rm Ran}\,}
\newcommand{\RRa}{{\rm RR}}
\newcommand{\rank}{{\rm rank}\,}
\renewcommand{\Re}{{\rm Re}\,}
\newcommand{\SO}{{\rm SO}\,}
\newcommand{\sa}{{\rm sa}}
\newcommand{\spa}{{\rm span}\,}
\newcommand{\tsr}{{\rm tsr}}
\newcommand{\Tr}{{\rm Tr}\,}

\newcommand{\CC}{{\mathbb C}}
\newcommand{\HH}{{\mathbb H}}
\newcommand{\QQ}{{\mathbb Q}}
\newcommand{\RR}{{\mathbb R}}
\newcommand{\TT}{{\mathbb T}}

\newcommand{\Ac}{{\mathcal A}}
\newcommand{\Bc}{{\mathcal B}}
\newcommand{\Cc}{{\mathcal C}}
\newcommand{\Fc}{{\mathcal F}}
\newcommand{\Hc}{{\mathcal H}}
\newcommand{\Ic}{{\mathcal I}}
\newcommand{\Jc}{{\mathcal J}}
\newcommand{\Kc}{{\mathcal K}}
\newcommand{\Lc}{{\mathcal L}}
\renewcommand{\Mc}{{\mathcal M}}
\newcommand{\Nc}{{\mathcal N}}
\newcommand{\Oc}{{\mathcal O}}
\newcommand{\Pc}{{\mathcal P}}
\newcommand{\Rc}{{\mathcal R}}
\newcommand{\Sc}{{\mathcal S}}
\newcommand{\Vc}{{\mathcal V}}
\newcommand{\Xc}{{\mathcal X}}
\newcommand{\Yc}{{\mathcal Y}}
\newcommand{\Wc}{{\mathcal W}}

\renewcommand{\gg}{{\mathfrak g}}
\newcommand{\hg}{{\mathfrak h}}
\newcommand{\kg}{{\mathfrak k}}

\newcommand{\ZZ}{\mathbb Z}
\newcommand{\NN}{\mathbb N}

\makeatletter
\title[Quasidiagonality of solvable Lie groups]{Quasidiagonality of
	$C^*$-algebras\\ of solvable Lie groups}
\author{Ingrid Belti\c t\u a and Daniel Belti\c t\u a}
\address{Institute of Mathematics ``Simion Stoilow'' 
of the Romanian Academy, 
P.O. Box 1-764, Bucharest, Romania}
\email{ingrid.beltita@gmail.com, Ingrid.Beltita@imar.ro}
\email{beltita@gmail.com, Daniel.Beltita@imar.ro}

\thanks{This work was supported by a grant of the Romanian National Authority for Scientific Research and
Innovation, CNCS--UEFISCDI, project number PN-II-RU-TE-2014-4-0370}
\makeatother

\begin{abstract} 
	We characterize the  solvable Lie groups of the form $\RR^m\rtimes\RR$, 
	whose $C^*$-algebras are quasidiagonal. Using this result, we determine
  the  connected simply connected solvable Lie groups of type~I whose $C^*$-algebras are strongly quasidiagonal. As a by-product, 
we give also  examples of amenable Lie groups with non-quasidiagonal $C^*$-algebras.

\end{abstract}

\maketitle



\section{Introduction}

Quasidiagonality properties have proved to be very important in the study of $C^*$-algebras and their applications. 
In the case of group $C^*$-algebras, the proof of Rosenberg's conjecture
was recently completed: 
The amenability of a countable discrete group is equivalent to quasidiagonality of its reduced $C^*$-algebra (see \cite{Hd87} and \cite{TiWhWi16}, and also \cite[Prop. VII.7.8]{Da96}). 
However, this statement does not carry over to general locally compact groups, and quasidiagonality properties of their $C^*$-algebras still remain to be understood. 
Interesting results in this direction were obtained in the paper \cite{SW01}. 

In the present paper we address the above problems for connected solvable Lie groups, continuing our quest for understanding topological aspects of their unitary dual (see  \cite{BB16a}, \cite{BB16b},  \cite{BBG16}, \cite{BBL17}).
We show that a connected simply connected solvable Lie group of type~I  has quasidiagonal $C^*$-algbra if and only if it is CCR.
We will also see below that there are  many connected simply connected solvable Lie groups of type~I   that have quasidiagonal $C^*$-algbras. We  also construct pretty large and precisely described classes of Lie groups which are amenable and yet their $C^*$-algebras are not quasidiagonal. 
More specifically, we characterize the generalized 
$ax+b$-groups, that is, solvable Lie groups of the form $\RR^m \rtimes\RR$, 
whose $C^*$-algebras are quasidiagonal. 
The $C^*$-algebras of these groups have been recently studied in \cite{LiLu13}.

\subsection*{Main results and structure of this paper}
Let us state our main result, which will be proved in Section~\ref{sect3}.

\begin{theorem}\label{main1}
	For any connected simply connected solvable Lie group $G$ of type~I with its Lie algebra $\gg$, the following properties are equivalent:
	\begin{enumerate}[(i)]
		\item\label{main1_strqsd} 
		The $C^*$-algebra $C^*(G)$ is strongly quasidiagonal. 
		\item\label{main1_ad} 
		For every $A\in\gg$ the eigenvalues of the linear map $\gg\ni X\mapsto [A,X]\in\gg$ are purely imaginary or zero.  
		\item\label{main1_CCR} 
		For every $[\pi]\in\widehat{G}$ one has $\pi(C^*(G))=\Kc(\Hc_\pi)$.  
	\end{enumerate}
	If $G$ is an exponential solvable Lie group, 
	then the above properties are further equivalent to 
	\begin{enumerate}[(i)]\setcounter{enumi}{3}
		\item\label{main1_nilp} 
		The Lie group $G$ is nilpotent. 
		\item\label{main1_nilp2} 
		The coadjoint orbits of $G$ are closed. 
	\end{enumerate}
\end{theorem}

The proof of Theorem~\ref{main1}  is based on quasidiagonality properties of some special generalized $ax+b$-groups.  

We give a thorough treatment of $C^*$-algebras of generalized $ax+b$-groups in Section~\ref{sect2}. 
Namely, we first establish necessary and sufficient conditions for groups in this class to have isomorphic $C^*$-algebras (see Theorems~\ref{isom-dir} and  \ref{isom-conv}).
We thus generalize the results of \cite{LiLu13} using a completely different method, based on topological equivalence of linear dynamical systems. 
Then we characterize the generalized $ax+b$-groups whose $C^*$-algebras are quasidiagonal (Theorem~\ref{ax+b}). 
As a by-product we show that among the isomorphism classes of the $C^*$-algebra of these amenable Lie groups there is exactly one class which contains non-quasidiagonal $C^*$-algebras (Corollary~\ref{ax+b_cor2}). 

As already mentioned, Section~\ref{sect3} is devoted to the proof of the above Theorem~\ref{main1}. 

In Section~\ref{sect4} we prove that if $G$ is a connected locally compact solvable group of type~I such that $C^*(G)$ has a faithful tracial state, then $G$ is commutative. 
This shows that the study of quasidiagonality of connected solvable Lie groups of type~I is in some sense complementary to the case of discrete groups as seen in \cite{TiWhWi16}, and therefore we had to use here completely different methods for establishing when their $C^*$-algebras are quasidiagonal. 

\subsection*{General notation and terminology} 
For any $C^*$-algebra $\Ac$ we define $\Ac^1:=\Ac$ if $\Ac$ is unital, 
while if $\Ac$ has no unit element, then $\Ac^1:= \Ac \oplus \CC \1\subseteq M(\Ac)$ is the unitization of $\Ac$, 
and $M(\Ac)$ is the multiplier algebra of~$\Ac$. 

We denote by $\widehat{\Ac}$ the set of equivalence classes of irreducible $*$-representations of $\Ac$, endowed with its usual topology. 
If $\pi\colon\Ac\to\Bc(\Hc_\pi)$ is an irreducible $*$-re\-pres\-entation, 
we denote its equivalence class by $[\pi]\in\widehat{\Ac}$. 
 
We say that a $*$-representation $\pi\colon \Ac\to\Bc(\Hc_\pi)$ of a separable  $C^*$-algebra~$\Ac$ is quasidiagonal, 
if one has  
$$(\forall a\in\Ac)\ \lim\limits_{n\to\infty}\Vert \pi(a)P_n-P_n\pi(a)\Vert=0$$
for a suitable sequence $P_n=P_n^*=P_n^2\in\Bc(\Hc_\pi)$ with  $\text{rank}\,P_n<\infty$ for every $n\ge 1$ and $P_n\to\1$ 
in the strong operator topology as $n\to\infty$. 
The separable $C^*$-algebra $\Ac$ is called quasidiagonal if it has a quasidiagonal faithful $*$-representation, 
and $\Ac$ is called strongly quasidiagonal if every irreducible 
$*$-representation of $\Ac$ is quasidiagonal
(see e.g., \cite{Vo93}, \cite{Da96}, \cite{BrOz08}). 

For $G$  a locally compact group, let
 $\widehat{G}$ be its unitary dual, 
that is, the set of all equivalence classes $[\pi]$ of unitary irreducible representations $\pi\colon G\to\Bc(\Hc_\pi)$. 
Then, as usually, we identify $\widehat{G}$ with $\widehat{\Ac}$, 
where $\Ac=C^*(G)$.

By $C^*$-dynamical system we mean a triple $(\Ac,G,\alpha)$, where $\Ac$ is a $C^*$-algebra, $G$ is a locally compact group, and $\alpha\colon G\times\Ac\to\Ac$, $(g,a)\mapsto\alpha_g(a)$ is a continuous action of $G$ by $*$-automorphisms of $\Ac$. 
Sometimes we denote a $C^*$-dynamical system simply as $G\times\Ac\to\Ac$, $(g,a)\mapsto g\cdot a$, especially when the notation for the action $\alpha$ is not essential.

In this paper by a Lie group we mean a finite dimensional real Lie group. 
We recall that a connected Lie group $G$ is solvable if its Lie algebra 
$\gg$ is solvable, that is, 
$$ \gg^{(n)}=\{0\} \quad \text{for $n$ large enough}, $$
where 
$$ \gg^{(0)}= \gg, \; \gg^{(k)}= [\gg^{(k-1)}, \gg^{(k-1)}] \; \text{for $k\ge 1$}.$$
This is further equivalent with the fact that $G$ is solvable as a discrete group. 

An exponential Lie group is a Lie group $G$ whose exponential map $\exp_G\colon\gg\to G$ is a bijection, where $\gg$ is the Lie algebra of~$G$. All exponential Lie groups are solvable. 
See for instance \cite{FuLu15} for more details.

\section{On the $C^*$-algebras of generalized $ax+b$-groups}
\label{sect2}

In this section, we fix a finite-dimensional real vector space~$\Vc$.  

For any     
$D\in \End(\Vc)$ we 
define the corresponding \emph{generalized $ax+b$-group} as 
the 
connected simply connected Lie group~$G_D:=\Vc\rtimes_{\alpha_D}\RR$,  
where the semidirect product of the abelian Lie groups $(\RR,+)$ and $(\Vc,+)$ involves the action 
$$\alpha_D\colon\Vc\times\RR\to\Vc, 
\quad \alpha_D(v,t):=\ee^{tD}v$$
so the group operation is given by $(v_1,t_1)\cdot(v_2,t_2)=(v_1+\ee^{t_1D}v_2,t_1+t_2)$ for all $v_1,v_2\in\Vc$ and $t_1, t_2\in\RR$. 

We also denote by $\Vc/\alpha_D$ the orbit space of the group action~$\alpha_D$, endowed with its quotient topology.

We consider the $\CC$-linear extension $D\colon\Vc_{\CC}\to\Vc_{\CC}$, 
and we denote its spectrum by $\sigma(D)$, 
where $\Vc_{\CC}:=\Vc\otimes_{\RR}\CC$. 
For  $\mu\in \sigma(D)\cap(\CC \setminus \RR)$, the real generalized eigenspace for $\mu$, $\overline{\mu}$ is the linear subspace of $\Vc$ given by 
$$ E^D(\mu):=\{v_1, v_2\in \Vc\mid v_1+ \ie v_2 \in \Ker(D-\mu \1)^m\},$$
where $m$ is the dimension of the largest Jordan block for $\mu$. 
(See \cite[Ex. 1.6.9]{CW14}.)
If  $\mu\in \sigma(D)\cap \RR$, then the real generalized eigenspace for $\mu$ is $E^D(\mu):=\Ker(D-\mu \1)^m$, 
where $m$ is again the dimension of the largest Jordan block for $\mu$. 
For every $\mu\in\CC\setminus\sigma(D)$ we also define  $E^D(\mu):=\{0\}$. 

We also define 
$\Vc^D_{\pm}:=\bigoplus\limits_{\pm\Re\mu>0}E^D(\mu)$ 
and $\Vc^D_0:=\bigoplus\limits_{\Re\mu=0}E^D(\mu)$, and then one has 
$\Vc^D=\Vc^D_{-}\oplus\Vc^D_0\oplus\Vc^D_{+}$. 
(See for instance \cite[\S 1.4]{CW14}.) 
With this notation we put  
\begin{equation}\label{npm}
n^D_{\pm}:=\dim\Vc^D_{\pm}.
\end{equation}

\subsection{Isomorphism classes of $C^*$-algebras of generalized $ax+b$-groups}

\begin{lemma}\label{Ladis}
	If  $D_1,D_2\in \End(\Vc)$, then the following assertions are equivalent: 
	\begin{enumerate}[{\rm(i)}]
		\item\label{Ladis_item1} 
		There exists a homeomorphism $\Phi\colon\Vc\to\Vc$ satisfying 
		$\Phi\circ\ee^{tD_1}=\ee^{tD_2}\circ\Phi$ for every $t\in\RR$. 
		\item\label{Ladis_item2} 
		One has $n^{D_1}_{\pm}=n^{D_2}_{\pm}$ and $T D_1\vert_{\Vc^{D_1}_0}=D_2T$ for a suitable linear isomorphism $T\colon\Vc^{D_1}_0\to\Vc^{D_2}_0$. 
	\end{enumerate}
\end{lemma}

\begin{proof}
	See \cite[Th. B']{La73}.
	\end{proof}

A proof of Lemma~\ref{Ladis} in the special case of hyperbolic maps 
(that is, under the additional assumption $\dim\Vc^{D_1}_0=\dim\Vc^{D_2}_0=\{0\}$) 
can also be found for instance in \cite[Th. 2.2.5]{CW14}.


\begin{theorem}\label{isom-dir}
	Let  $D\in \End(\Vc)$ and define
	$$\widetilde{D}=
	\begin{pmatrix}
	-\1_{n_{-}} & 0 & 0 \\
	0 & D_0 & 0 \\
	0 & 0 & \1_{n_{+}}
	\end{pmatrix}$$
	with respect to the direct sum decomposition $\Vc=\Vc^{D}_{-}\oplus\Vc^{D}_0\oplus\Vc^{D}_{+}$, 
	where $n_\pm:=n^{D}_\pm$, we have denoted by $\1_{n_\pm}$ the identity map of $\Vc^{D}_\pm$, and $D_0:=D\vert_{\Vc^{D}_0}$. 	 
	
	Then the $C^*$-algebras $C^*(G_D)$ and $C^*(G_{\widetilde{D}})$ are $*$-isomorphic. 
\end{theorem}

\begin{proof} 
	We denote $\Wc:=\Vc^*$ and also let $T$ be the identity map on $\Vc^{D}_0$.
	 
	Then one has 	
	$\dim\Wc^{D^*}_{\pm}=\dim\Wc^{{\widetilde{D}}}_{\pm}$ and $T^* {\widetilde{D}}^*=D^*T^*$,  where $T^*\colon\Wc^{{\widetilde{D}}^*}_0\to\Wc^{D^*}_0$ is a linear isomorphism. 
	Then, by Lemma~\ref{Ladis}, 
	there exists a homeomorphism $\Phi\colon\Vc^*\to\Vc^*$ satisfying 
	$\Phi\circ\ee^{tD^*}=\ee^{t{\widetilde{D}}^*}\circ\Phi$ for every $t\in\RR$. 
	The map 
	$$\varphi\colon\Cc_0(\Vc^*)\to\Cc_0(\Vc^*),\quad \phi(f):=f\circ\Phi$$
	is an equivariant isomorphism between the $C^*$-dynamical systems $(\Cc_0(\Vc^*),\RR,\alpha^{D})$ and $(\Cc_0(\Vc^*),\RR,\alpha^{{\widetilde{D}}})$, 
	where $\alpha^{D}\colon\RR\times\Cc_0(\Vc^*)\to\Cc_0(\Vc^*)$, 
	$\alpha^{D}(t,f):=f\circ\ee^{tD^*}$, 
	and $\alpha^{{\widetilde{D}}}$ is similarly defined with $D$ replaced by $\widetilde{D}$.
	 
	We now obtain a $*$-isomorphism  
	$$\id\ltimes\varphi\colon
	\RR\ltimes_{\alpha^{D}}\Cc_0(\Vc^*)\to
	\RR\ltimes_{\alpha^{{\widetilde{D}}}}\Cc_0(\Vc^*)$$ 
	by \cite[Lemma 2.65]{Wi07}. 
	We now recall from \cite[Ex. 3.16]{Wi07} that there are $*$-isomorphisms 
	$C^*(G_D)\simeq \RR\ltimes_{\alpha^{D}}\Cc_0(\Vc^*)$ 
	and 
	$C^*(G_{\widetilde{D}})\simeq \RR\ltimes_{\alpha^{{\widetilde{D}}}}\Cc_0(\Vc^*)$, 
	and we are done. 
\end{proof}

\begin{remark}
\normalfont
Theorem~\ref{isom-dir} is equivalent to the following fact: 
If $D_1,D_2\in \End(\Vc)$ with $n^{D_1}_{\pm}=n^{D_2}_{\pm}$ and $T D_1=D_2T$ for a suitable linear isomorphism $T\colon\Vc^{D_1}_0\to\Vc^{D_2}_0$, 
then the $C^*$-algebras $C^*(G_{D_1})$ and $C^*(G_{D_2})$ are $*$-isomorphic. 

We point out that a related idea was suggested in \cite{Ro76}.
\end{remark}

We now turn to establishing a converse to Theorem~\ref{isom-dir}. 
The following lemma replaces the method of coadjoint orbits for the particular class of solvable Lie groups we are interested in here.

\begin{lemma}\label{isom-conv-0}
	Let $D\in\End(\Vc)$. 
	If $\Vc^D_0=\Ker D$, then $G_D$ is an exponential solvable Lie  group and its unitary dual is homeomorphic to the quotient space 
	$(\Vc^*\times\RR)/\sim$, where 
	we define an equivalence relation $\sim$ on $\Vc^*\times\RR$ 
	whose equivalence classes are given by 
	$$[(\xi,s)]=\begin{cases}
	\{(\xi,s)\} & \text{ if }\xi\in(\Vc^D_{-}\oplus\Vc^D_{+})^\perp\simeq(\Vc^D_0)^*,\\
	\{\ee^{tD^*}\xi\mid t\in\RR\}\times\RR & \text{ if }\xi\in\Vc^*\setminus(\Vc^D_{-}\oplus\Vc^D_{+})^\perp
	\end{cases}$$ 
	then the unitary dual space of $G_D$ is homeomorphic to~$(\Vc^*\times\RR)/\sim$. 
\end{lemma}

\begin{proof}
	The hypothesis $\Vc^D_0=\Ker D$ implies that $\sigma(D)\cap\ie\RR\subseteq\{0\}$, hence $G_D$ is an exponential solvable Lie group by an application of \cite[Def. 5.2.11 and Th. 5.2.16]{FuLu15}. 
	In particular, $C^*(G_D)$ is type~I. 
	Using \cite[Th.~6.2 and Th.~8.43]{Wi07} 
	it then follows that the orbits 
	of $\alpha^D\colon \RR\times\Vc^*\to\Vc^*$ are locally closed subsets 
	of~$\Vc^*$. 
	Moreover, for any $\xi\in\Vc^*$ its stability group $H_\xi$ with respect to the group action $\alpha^D$ is a closed connected subgroup of $H:=(\RR,+)$ by \cite[Th. 5.3.2 and Prop. 5.2.13]{FuLu15}, hence that stability group is equal to 
	$\RR$ if $\xi\in(\Vc^D_{-}\oplus\Vc^D_{+})^\perp\simeq(\Vc^D_0)^*$ and is equal to $\{0\}$ if $\xi\in\Vc^*\setminus(\Vc^D_{-}\oplus\Vc^D_{+})^\perp$. 
	Here we have used again the hypothesis  $\Vc^D_0=\Ker D$. 
	We then obtain 
	$$H_\xi^\perp=
	\begin{cases}
	\{0\} & \text{ if } \xi\in(\Vc^D_{-}\oplus\Vc^D_{+})^\perp\\
	\RR  & \text{ if } \xi\in\Vc^*\setminus(\Vc^D_{-}\oplus\Vc^D_{+})^\perp.
	\end{cases}
	$$
	It then follows  by \cite[Th.~8.3.9]{Wi07} that $(\Vc^*\times\RR)/\sim$ is 
	homeomorphic to the primitive ideal space of $\RR \ltimes_{\alpha^{D}}\Cc_0(\Vc^*)$, 
	which is further homeomorphic to 
	the unitary dual space of $G_D$, and this completes the proof. 
\end{proof}

\begin{lemma}\label{traj}
	Let  $D\in\End(\Vc)$ and denote $\Vc_\epsilon:=\Vc^D_\epsilon$ 
	for $\epsilon\in\{-,0,+\}$. 
	There exist a norm $\Vert\cdot\Vert$ on $\Vc$ and a constant $a>0$ with the properties:
	\begin{enumerate}[{\rm(i)}]
		\item\label{traj_item1} 
		If $v\in\Vc_{-}$ then $\Vert \ee^{tD} v\Vert\le \ee^{-at}\Vert v\Vert$ and $\Vert \ee^{-tD} v\Vert\ge \ee^{at}\Vert v\Vert$ for all $t\ge0$. 
		\item\label{traj_item2} 
		If $v\in\Vc_{+}$ then $\Vert \ee^{tD} v\Vert\ge \ee^{at}\Vert v\Vert $ and $\Vert \ee^{-tD} v\Vert\le \ee^{-at}\Vert v\Vert$ for all $t\ge0$. 
		\item\label{traj_item3} 
		For every $v\in\Vc_{-}\setminus\{0\}$ 
		(respectively, $v\in\Vc_{+}\setminus\{0\}$) the function $\RR\to(0,\infty)$, 
		$t\mapsto\Vert \ee^{tD}v\Vert$ is strictly decreasing 
		(respectively, increasing) and bijective. 
	\end{enumerate}
\end{lemma}

\begin{proof}
	Using \cite[Prop. 2.2.7 and proof of Prop. 2.2.8]{CW14} 
	we first define $\Vert\cdot\Vert$ on $\Vc_{-}$ and on $\Vc_{+}$ satisfying the conditions in the statement. 
	Then we define $\Vert\cdot\Vert$on $\Vc_0$ as an arbitrary norm, 
	and finally we define 
	$\Vert v_{-}+v_0+v_{+}\Vert:=\max(\Vert v_{-}\Vert,\Vert v_0\Vert,\Vert v_{+}\Vert)$ for any $v_{\pm}\in\Vc_{\pm}$ and $v_0\in\Vc_0$, using the direct sum decomposition $\Vc=\Vc_{-}\oplus\Vc_0\oplus\Vc_{+}$. 
	This completes the proof of Assertions~\eqref{traj_item1}--\eqref{traj_item2}. 
	
	For~\eqref{traj_item3} we note that if $v\in\Vc_{-}\setminus\{0\}$ then by \eqref{traj_item1} the function $h_v\colon\RR\to[0,\infty)$, 
	$t\mapsto\Vert \ee^{tD}v\Vert$, is strictly decreasing 
	and $\lim\limits_{t\to\infty}h_t(v)=0$ while  $\lim\limits_{t\to 0}h_t(v)=\infty$. 
	On the other hand, if $v\in\Vc_{+}\setminus\{0\}$ then by \eqref{traj_item2} the function $h_v$ is strictly increasing 
	and $\lim\limits_{t\to\infty}h_t(v)=\infty$ while  $\lim\limits_{t\to 0}h_t(v)=0$. 
\end{proof}

\begin{lemma}\label{stratif}
	Let $D\in \End(\Vc)$ and denote $\Vc_\epsilon:=\Vc^D_\epsilon$ 
	for $\epsilon\in\{-,0,+\}$. 
	
	Assume $\Vc_0=\Ker D$ and 
	let $q\colon\Vc\to\Vc/\alpha_D$ denote the quotient map associated to the group action~$\alpha_D$. 
	We also define 
	$$A:=\Vc\setminus((\Vc_{-}+\Vc_0)\cup(\Vc_{+}+\Vc_0)) 
	\text{ and }A_{\pm}:=(\Vc_\pm+\Vc_0)\setminus\Vc_0.$$ 
	Then the following assertions hold: 
	\begin{enumerate}[{\rm(i)}]
		\item\label{stratif_item1} 
		One has the partition 
		$\Vc=A\sqcup A_{-}\sqcup A_{+}\sqcup\Vc_0$ into invariant subsets with respect to the group action $\alpha_D$.
		\item\label{stratif_item2} 
		If $A\ne\emptyset$, then $q(A)$ is open, dense and Hausdorff in $\Vc/\alpha_D$. 
		\item\label{stratif_item3} 
		The sets $q(A_+)$ and $q(A_{-})$ are closed subsets of  $(\Vc/\alpha_D)\setminus q(\Vc_0)$. 
		\item\label{stratif_item4} 
		The set $q(A_\pm)$ is homeomorphic to the Cartesian product of $\Vc_0$ and unit sphere of $\Vc_\pm$. 
		\item\label{stratif_item5} If $\Vc_\pm=\Ker(D\mp\1)$, then the set  $q(A)$ is the set of all separated points of~$\Vc/\alpha_D$.
	\end{enumerate}
\end{lemma}

\begin{proof}
	Assertion~\eqref{stratif_item1} is clear, 
	using the direct sum decomposition $\Vc=\Vc_{-}\oplus\Vc_0\oplus\Vc_{+}$, 
	in which all summands are invariant to the group action~$\alpha$. 
	
	For \eqref{stratif_item2}--\eqref{stratif_item3} we recall that the quotient map 
	$q$ is continuous and open, while $A$ is open and dense in $\Vc$ if $A\ne\emptyset$. 
	Hence $q(A)$ is open and dense in $\Vc/\alpha_D$. 
	Similarly, we note that $A_\pm$ are disjoint open subsets of $\Vc\setminus(A\sqcup\Vc_0)$, hence 
	$q(A_+)$ and $q(A_{-})$ are closed subsets of  $(\Vc/\alpha_D)\setminus q(\Vc_0)$. 
	
	It remains to prove that $q(A)$ is Hausdorff and that 
	\eqref{stratif_item4} and \eqref{stratif_item5} hold true. 
	To this end we note that   
	$$A=\{v_{-}+v_0+v_{+}\in \Vc_{-}\oplus\Vc_0\oplus\Vc_{+}\mid v_{-}\ne0\ne v_{+}\}$$
	and 
	$$A_\pm=\{v_\pm+v_0\in \Vc_\pm\oplus\Vc_0\mid v_\pm\ne0\}.$$
	We endow $\Vc$ with the norm given by Lemma~\ref{traj} 
	and we define 
	$$\Delta:=\{v_{-}+v_0+v_{+}\in\Vc_{-}\oplus\Vc_0\oplus\Vc_{+}
	\mid \Vert v_{-}\Vert=\Vert v_{+}\Vert\ne0\}$$
	and 
	$$\Delta_\pm=\{v_\pm+v_0\in \Vc_\pm\oplus\Vc_0\mid \Vert v_\pm\Vert=1\}.$$
	Then the maps 
	$$\Psi:=\alpha\vert_{\RR\times\Delta}\colon \RR\times\Delta\to A
	\text{ and }
	\Psi_\pm:=\alpha\vert_{\RR\times\Delta_\pm}\colon \RR\times\Delta_\pm\to A_\pm$$
	are homeomorphisms as a direct consequence of Lemma~\ref{traj}\eqref{traj_item3}. 
	One thus obtains the homeomorphisms $q\vert_\Delta\colon\Delta\to q(A)$  and $q\vert_{\Delta_\pm}\colon\Delta_\pm\to q(A_\pm)$, 
	which entail~\eqref{stratif_item4} and the fact that the open set $q(A)\subseteq\Vc/\alpha_D$ is Hausdorff in its relative topology. 
	Therefore every point of $q(A)$ is separated in $\Vc/\alpha_D$. 
	
	It remains to check that no point in $(\Vc/\alpha_D)\setminus q(A)$ is separated in $\Vc/\alpha_D$. 
	One has $\Vc/\alpha_D=q(A)\sqcup q(A_{-})\sqcup q(A_{+})\sqcup q(\Vc_0)$. 
	 To this end we note that, for arbitrary $v_{\pm}\in\Delta_\pm\cap\Vc_\pm$ and $v_0\in\Vc_0$,  
	the sequence 
	$$q(\ee^{-j}v_{-}+v_0+v_{+})
	=\{\ee^{-j}\ee^{-t}v_{-}+v_0+\ee^tv_{+}\mid t\in\RR\} $$
	contains in its limit set both $v_{-}+v_0$ (for $t_j=-j$) 
	and $v_0+v_{+}$ (for $t_j=0$), hence these points cannot be separated from each other by disjoint open neighborhoods. 
	This completes the proof. 
\end{proof}

\begin{theorem}\label{isom-conv}
	Let  $D_1,D_2\in \End(\Vc)$ for which 
	the $C^*$-algebras $C^*(G_{D_1})$ and $C^*(G_{D_2})$ are $*$-isomorphic and $\Vc^{D_j}_0=\Ker D_j$ for $j=1,2$. 
	Then 
	$\dim\Vc^{D_1}_0=\dim\Vc^{D_2}_0$ and 
	$\{n^{D_1}_\epsilon\mid\epsilon=\pm\}=\{n^{D_2}_\epsilon\mid\epsilon=\pm\}$. 
\end{theorem}

\begin{proof}
	The solvable Lie group $G_{D_j}$ is exponential by Lemma~\ref{isom-conv-0}.  
	It then follows by \cite[Th. 3.5]{BB16a} that the real rank of $C^*(G_{D_j})$ can be computed in the following way: 
	$$\begin{aligned}
	\RRa(C^*(G_{D_j}))
	&=\dim(\gg_{D_j}/[\gg_{D_j},\gg_{D_j}])
	=\dim\gg_{D_j}-\dim(\Ran D_j) \\
	& =1+\dim\Vc-\dim(\Ran D_j)
	=1+\dim(\Ker D_j).
	\end{aligned}$$ 
	Since $C^*(G_{D_1})\simeq C^*(G_{D_2})$, we obtain $\dim(\Ker D_1)=\dim(\Ker D_2)$, hence by hypothesis $\dim\Vc^{D_1}_0=\dim\Vc^{D_2}_0$. 
	
	It remains to prove that $n^1_\pm=n^2_\pm$, where 
	$n^j_\pm:=n^{D_j}_\pm$. 
	By Theorem~\ref{isom-dir}, we may assume $\Vc^{D_j}_\pm=\Ker(D_j\mp\1)$, that is, one has 
	$$D_j=
	\begin{pmatrix}
	-\1_{n^j_{-}} & 0 & 0 \\
	0 & 0 & 0 \\
	0 & 0 & \1_{n^j_{+}}
	\end{pmatrix}$$
	with respect to the direct sum decomposition $\Vc=\Vc^{D_j}_{-}\oplus\Vc^{D_j}_0\oplus\Vc^{D_j}_{+}$, 
	where $\1_{n^j_\pm}$ is the identity map of $\Vc^{D_j}_\pm$ 
	for $j=1,2$.	
	
	Then we use the fact that the unitary dual spaces of $G_{D_1}$ and $G_{D_2}$ are homeomorphic, 
	and by Lemmas \ref{isom-conv-0} and \ref{stratif} one obtains $\{n^1_\epsilon\mid\epsilon=\pm\}=\{n^2_\epsilon\mid\epsilon=\pm\}$ by a direct application of an argument in the proof of \cite[Prop. 2.1]{LiLu13}. 
\end{proof}

\begin{remark}
	\normalfont
	Theorems~\ref{isom-dir} 
	and \ref{isom-conv} 
	are generalizations of \cite[Prop. 2.4]{LiLu13}, 
	and of \cite[Prop. 2.1]{LiLu13}, respectively, 
	where one additionally assumed that $D_j\in\End(\Vc)$ is semisimple (which is stronger than our assumption  $\Vc^{D_j}_0=\Ker D_j$) for $j=1,2$.  
\end{remark}

\subsection{Quasidiagonality of generalized $ax+b$-groups}\label{sect7}

The main results of this section are Theorem~\ref{ax+b} and its corollaries, and their proofs require some basic notions on topological dynamical systems, which we now recall. 

Let $X\times\RR\to\RR$, $(x,t)\mapsto x\cdot t$, be a continuous right action of the group $(\RR,+)$ on a compact space~$X$. 
For every subset $Y\subseteq X$ we define the closed subsets of $X$ 
$$\begin{aligned}
\omega(Y)
&:=\bigcap\limits_{t\in\RR}\Cl(Y\cdot[t,\infty)) \\
\omega^*(Y)
&:=\bigcap\limits_{t\in\RR}\Cl(Y\cdot(-\infty,t]))
\end{aligned}$$
where $\Cl(\bullet)$ stands for the closure of a subset of~$X$. 
Equivalently
$$\omega(Y)=\{x\in X\mid(\exists \{y_n\}_{n\in\NN}\subseteq Y,\,\{t_n\}_{n\in\NN}\subseteq\RR)\ 
\lim\limits_{n\to\infty} t_n=\infty,\ 
\lim\limits_{n\to\infty}y_n\cdot t_n=x\}  $$
and there is a similar description of $\omega^*(Y)$ in which $\lim\limits_{n\to\infty} t_n=-\infty$.

For every singleton set $\{x\}\subseteq X$ we denote $\omega(\{x\})=:\omega(x)$ and $\omega^*(\{x\})=:\omega^*(x)$. 

A subset $A\subseteq T$ is said to be an \emph{attractor},  
respectively \emph{repeller}, 
if it has a neighborhood $N$ with $\omega(N)=A$, respectively $\omega^*(N)=A$. 
It is clear that attractors and repellers are closed invariant subsets of $X$. 

If $A\subseteq X$ is an attractor, then we define 
its corresponding \emph{complementary repeller} 
$$A^*:=\{x\in X\mid\omega(x)\cap A=\emptyset\}=X\setminus\widetilde{\alpha}(A)$$
(see \cite[Lemma 8.2.5]{CW14})
where the set 
$\widetilde{\alpha}(A):=\{x\in X\mid\omega(x)\cap A\ne\emptyset\}$ 
is called the \emph{domain of influence} of~$A$. 
The pair of sets $(A,A^*)$ is said to be a \emph{nontrivial  attractor-repeller pair} if $A\cup A^*\subsetneqq X$. 
(We point out that an attractor-repeller pair $(A,A^*)$ was called trivial in \cite[page 161]{CW14} if either $A=X$ and $A^*=\emptyset$, or $A=\emptyset$ and $A=X$, which leads to a different notion of nontrivial attractor-repeller pair. 
If however $X$ is connected then, using the fact that both $A$ and $A^*$ are closed, it follows that our notion of nontriviality agrees with the terminology of \cite{CW14}. 
Incidentally, in our applications of these notions in Lemma~\ref{img} and Proposition~\ref{attr} below, the space $X$ is the one-point compactification of a finite-dimensional real vector space hence is connected.)

\begin{notation}\label{compactification}
	\normalfont 
		For any locally compact space $S$ we denote by $\overline{S}:=S\sqcup\{\infty\}$ its one-point compactification if $S$ is non-compact, 
	and for every homeomorphism $\theta\colon S\to S$ we denote again by $\theta\colon \overline{S}\to\overline{S}$ 
	its extension to a homeomorphism with $\theta(\infty)=\infty$. 
	If $S$ is compact, then we denote  $\overline{S}:=S$ 
   and $\Cc_0(S):=\Cc(S)$. 
\end{notation}

\begin{lemma}\label{crossunit}
	Let $\Mc_0$ be a separable $C^*$-algebra without unit, with its unitization $\Mc:=\Mc_0\oplus\CC\1$. 
	Assume that $G$ is an amenable separable locally compact group for which $C^*(G)$ is quasidiagonal. 
	If one has a continuous action 
	$G\times\Mc_0\to\Mc_0$, canonically extended to an action $G\times\Mc\to\Mc$,  
	then $G\ltimes\Mc_0$ is quasidiagonal if and only if $G\ltimes\Mc$ is quasidiagonal. 
\end{lemma}

\begin{proof} 
	One has the split exact sequence $0\to\Mc_0\to\Mc\to\CC\1\to 0$ 
	which, by \cite[Lemma~2.82]{Ph87}, leads to the split exact sequence 
	$$0\to G\ltimes \Mc_0\to G\ltimes\Mc\to C^*(G)\to 0.$$ 
	If $G\ltimes\Mc$ is quasidiagonal then also its ideal $G\ltimes\Mc_0$ is quasidiagonal. 
	
	Conversely, let us assume that $G\ltimes\Mc_0$ is quasidiagonal. 
	The $C^*$-algebra $G\ltimes \Mc_0$ is separable, hence is $\sigma$-unital. 
	Moreover, in the above split exact sequence, $C^*(G)$ is a separable nuclear $C^*$-algebra since 
	$G$ is an amenable separable locally compact group (see \cite[p.~393]{Pe79}). 
	Since $C^*(G)$ is quasidiagonal, 
	it then follows by \cite[Prop. 2.5]{BrDa04} that $G\ltimes\Mc$ is quasidiagonal, and this completes the proof. 
\end{proof}

\begin{proposition}\label{comp}
	Let $\alpha\colon S\times\RR\to\RR$, $(x,t)\mapsto x\cdot t$, be a continuous right action of the group $(\RR,+)$ on a separable, locally compact, metrizable space $S$. 
	We extend this action to $\alpha\colon \overline{S}\times\RR\to\overline{S}$ as above. 
	Then the following assertions are equivalent: 
	\begin{enumerate}[\rm(i)]
		\item\label{comp_item1} The flow $\alpha$ on $\overline{S}$ has no nontrivial  attractor-repeller pair. 
		\item\label{comp_item2} The $C^*$-algebra $\RR\ltimes_{\alpha}\Cc_0(S)$ is quasidiagonal. 
		\item\label{comp_item3} There exists an embedding of $\RR\ltimes_{\alpha}\Cc_0(S)$ into an AF-algebra. 
	\end{enumerate}
\end{proposition}

\begin{proof}
Since the locally compact space $S$ is metrizable and separable, 
the compact space $\overline{S}$ is also metrizable. 
On the other hand, let $\Rc:=\bigcap\limits_A(A\sqcup A^*)$, where the intersection is taken over 
all attractors $A$ 
of the flow $\alpha$ on~$\overline{S}$. 
It follows by \cite[Def. 3.1.7 and Th. 8.3.3]{CW14} that the flow $\alpha$ on $\overline{S}$ is chain recurrent 
if and only if $\Rc=\overline{S}$. 
By the above definition of $\Rc$, the equality  
 $\Rc=\overline{S}$ holds if and only if 
for every attractor $A$ 
of~$\overline{S}$ one has $A\sqcup A^*=\overline{S}$, 
which by our definition prior to Notation~\ref{compactification} 
means that the attractor-repeller pair $(A,A^*)$ is trivial.  
Thus, the flow $\overline{S}$ has no nontrivial  attractor-repeller 
pair if and only if it is chain-recurrent. 
Therefore, by \cite[Th. 4.7]{Pi99}, 
the following conditions are equivalent: 
\begin{enumerate}
	\item\label{comp_item11} The flow $\alpha$ on $\overline{S}$ has no nontrivial  attractor-repeller pair. 
	\item\label{comp_item22} The $C^*$-algebra $\RR\ltimes_{\alpha}\Cc(\overline{S})$ is quasidiagonal. 
	\item\label{comp_item33} There exists an embedding of $\RR\ltimes_{\alpha}\Cc(\overline{S})$ into an AF-algebra. 
\end{enumerate}
By Lemma~\ref{crossunit}, one has \eqref{comp_item2}$\Leftrightarrow$\eqref{comp_item22}. 
Since $\RR\ltimes_{\alpha}\Cc(S)$ is an ideal of 
$\RR\ltimes_{\alpha}\Cc(\overline{S})$, 
we obtain \eqref{comp_item33}$\Rightarrow$\eqref{comp_item3}, 
and moreover it is well known that every AF-embeddable $C^*$-algebra is quasidiagonal, 
hence \eqref{comp_item3}$\Rightarrow$\eqref{comp_item2}. 
This completes the proof that actually 
\eqref{comp_item3}$\Leftrightarrow$\eqref{comp_item2}$\Leftrightarrow$\eqref{comp_item1}=\eqref{comp_item11}$\Leftrightarrow$\eqref{comp_item22}$\Leftrightarrow$\eqref{comp_item33}. 
\end{proof}

\begin{remark}
	\normalfont
	Here we collect a few basic facts to be used in the proof of Lemma~\ref{img} and  Proposition~\ref{attr} below.
	For more details, see \cite{CW14}. 
	
	Let $\Vc$ be a finite-dimensional real vector space and $D\in \End(\Vc)$ and let $E(\mu_k):=E^D(\mu_k)$, $k=1, \dots, r$ be the real generalized  eigenspaces for $D$, where  $\mu_k$, $k=1, \dots, r$,  are the eigenvalues of~$\sigma(D)$,

We denote by   $\lambda_1>\cdots>\lambda_\ell$ the distinct real parts of the eigenvalues $\mu_k$, $k=1, \dots, r$,  of~$\sigma(D)$, and define the Lyapunov space 
$\Vc(\lambda_j)$ to be the direct sum  of all real generalized eigenspaces $E(\mu_k)$  associated to eigenvalues $\mu_k$  with $\Re\mu_k=\lambda_j$.
Then one has  the direct sum decomposition 
\begin{equation}\label{attr_proof_eq0}
\Vc=\Vc(\lambda_1)\oplus\cdots\oplus\Vc(\lambda_\ell)
\end{equation} 
(see \cite[page 13]{CW14}). 
 One can prove   that
	\begin{equation}\label{attr_proof_eq-1}
	\Vc(\lambda_j)=\{0\}\cup\Bigl\{v\in\Vc\mid  \lim_{t\to\pm\infty}\frac{1}{t}\log\Vert\exp(tD)v\Vert=\lambda_j\Bigr\}
	\end{equation} 
	(see \cite[Th. 1.4.3]{CW14}). 
	 With  
	$\Vc_j:= \Vc(\lambda_j)\oplus\cdots\oplus\Vc(\lambda_\ell)$ 
	and $\Wc_j:=\Vc(\lambda_1)\oplus\cdots\oplus\Vc(\lambda_j)$, 
	it follows by \cite[Th. 1.4.4]{CW14} that if $v\in\Vc\setminus\{0\}$,  
	then 
	\begin{equation}\label{attr_proof_eq1}
	\lim_{t\to\infty}\frac{1}{t}\log\Vert\exp(tD)v\Vert=\lambda_j\iff v\in\Vc_j\setminus\Vc_{j+1}
	\end{equation}
	and 
	\begin{equation}\label{attr_proof_eq2}
	\lim_{t\to-\infty}\frac{1}{-t}\log\Vert\exp(tD)v\Vert=-\lambda_j\iff v\in\Wc_j\setminus\Wc_{j-1}, 
	\end{equation}
	where $\Vc_{\ell+1}=\Wc_0:=\{0\}$. 
	
	By \eqref{attr_proof_eq-1}, 
	one obtains 
	\begin{equation}\label{attr_proof_eq3}
	v\in\Vc(\lambda_j),\ \pm\lambda_j> 0
	\implies
	\lim\limits_{t\to\pm\infty}\exp(tD)v=\infty  
	\text{ and }\lim\limits_{t\to\mp\infty}\exp(tD)v=0. 
	\end{equation}
	It follows by
	\eqref{attr_proof_eq1}--\eqref{attr_proof_eq2}
	that
	\begin{equation}\label{attr_proof_eq5}
	v\in\Vc_j\setminus\Vc_{j+1},\ \lambda_j> 0
	\implies
	\lim\limits_{t\to\infty}\exp(tD)v=\infty 
	\end{equation}
	and 
	\begin{equation}\label{attr_proof_eq6}
	v\in\Wc_j\setminus\Wc_{j-1},\ \lambda_j< 0
	\implies
	\lim\limits_{t\to-\infty}\exp(tD)v=\infty. 
	\end{equation}
\end{remark}

\begin{lemma}\label{img}
	If $\Vc$ is a finite-dimensional real vector space 
	and $D\in \End(\Vc)$ whose  corresponding flow 
	$\alpha_D\colon \overline{\Vc}\times\RR\to\overline{\Vc}$,  $\alpha_D(v,t):=\exp(tD)v$, 
	has a nontrivial attractor-repeller pair, then there exists no $\tau\in\RR\setminus\{0\}$ with $\ie\tau\in\sigma(D)$. 
\end{lemma}

\begin{proof}
	Let $\overline{\Vc}=A\sqcup Z\sqcup A^*$ with $Z\ne\emptyset$, where $(A,A^*)$ is a nontrivial attractor-repeller pair, 
	and assume that there exists $\tau\in\RR\setminus\{0\}$ with  $\ie\tau\in\sigma(D)$. 
	We will argue by contradiction, in a few steps, 
	repeatedly using the fact that by \cite[Prop. 8.2.7]{CW14},
	\begin{equation}\label{disj}
	\omega(v)\subseteq A\text{ and } \omega^*(v)\subseteq A^*,   
	\text{ hence }\omega(v)\cap\omega^*(v)=\emptyset, 
	\text{ for all }v\in Z. 
	\end{equation}
	This implies that, since $0$ and $\infty$ are fixed points of the flow $\alpha_D$, we cannot have $0,\infty\in Z$ 
	hence $Z\subseteq\Vc\setminus\{0\}$. 
	
	\textit{Step 1:} We first prove that $\Vc=\Vc(0)$, that is, $\sigma(D)\subseteq\ie \RR$. 
	
	To this end we note that since $\ie\tau\in\sigma(D)$, there exist $v_1,v_2\in\Vc$ with $0\ne v_1+\ie v_2\in\Ker(D-\ie\tau)\subseteq\CC\otimes_{\RR}\Vc$. 
	Since $D(\Vc)\subseteq\Vc$, it is easy to check that $v_1,v_2\in\Vc\setminus\{0\}$ are linearly independent vectors. 
	It also follows that 
	$\exp(tD)v_1=(\cos(t\tau))v_1-(\sin(t\tau))v_2$ for every $t\in\RR$, 
	and, since $\tau\ne 0$, this directly implies 
	$\omega(v_1)=\omega^*(v_1)=\exp(\RR D)v_1\simeq\TT$.
	Then, by \eqref{disj}, one has $v_1\not\in Z$. 
	
	Therefore $v_1\in A\sqcup A^*$, and one similarly obtains 
	$sv\in A\sqcup A^*$ for every $s\in\RR$. 
	(If $s=0$ then $\omega(0)=\omega^*(0)=\{0\}$, so $0\not\in Z$ by the above argument.) 
	Thus $\RR v_1\subseteq A\sqcup A^*$, and since $\RR v_1$ is connected, we must have either $\RR v_1\subseteq A$ 
	or $\RR v_1\subseteq A^*$. 
	We may assume 
	$\RR v_1\subseteq A$. 
	Then, since $A$ is a closed subset of $\overline{\Vc}$, 
	we obtain  
	$\RR v_1\cup\{\infty\}\subseteq A$, hence $0,\infty\in A$. 
	On the other hand, if there exists $\lambda\in\sigma(D)$ with $\Re\lambda\ne 0$, 
	then \eqref{attr_proof_eq3} implies that we have either 
	$0\in A$ and $\infty\in A^*$, or $0\in A^*$ and $\infty\in A$, 
	which is a contradiction with $0,\infty\in A$.  
	Consequently, $\Re\lambda=0$ for every $\lambda\in\sigma(D)$, 
	that is, $\Vc=\Vc(0)$. 
	
	\textit{Step 2:} 
	We now prove that the map $D$ is semisimple. 
	
	To this end let us assume that $D$ is not semisimple. 
	We fix a real Jordan basis in $\Vc$ 
	(see for instance \cite[Th. 1.2.3]{CW14}), and then we define the Euclidean scalar product on $\Vc$ for which that Jordan basis is an orthonormal basis. 
	Let $D=S+N$ be the Jordan decomposition, where $S$ is semisimple, $N$ is nilpotent, and $SN=NS$. 
	The assumption that $D$ has only purely imaginary eigenvalues implies that for every $t\in\RR$ the map $\exp(tS)$ is an isometry, 
	while the assumption that $D$ is not semisimple implies that there exists an integer $k\ge 2$ with $N^k=0\ne N^{k-1}$. 
	Since $Z$ is an open nonempty subset of~$\Vc$ while $N^{k-1}\ne 0$, it follows that $Z\not\subset\Ker N^{k-1}$, 
	hence there exists $v\in Z$ with $N^{k-1}v\ne 0$. 
	Then it is easily checked that 
	$$\exp(tN)v=\sum\limits_{j=0}^{k-1}\frac{t^j}{j!}N^jv\to\infty\in\overline{\Vc}\text{ as }t\to\pm\infty$$
	hence 
	$$\lim_{t\to\pm\infty}\Vert \exp(tD)v\Vert
	=\lim_{t\to\pm\infty}\Vert \exp(tS)\exp(tN)v\Vert
	=\lim_{t\to\pm\infty}\Vert \exp(tN)v\Vert=\infty$$
	and this shows that $\infty\in\omega(v)\cap\omega^*(v)$, 
	which is a contradiction with \eqref{disj}. 
	Consequently, $D$ must be semisimple.  
	
	\textit{Step 3:} We now show that if $D\in\End(\Vc)$ is semisimple and $\sigma(D)\subseteq\ie\RR$, then 
	\begin{equation}\label{img_eq}
	(\forall v\in \Vc)\quad \omega(v)=\omega^*(v). 
	\end{equation} 
	For $v\in Z$, this is a contradiction with \eqref{disj}, 
	which completes the proof. 
	
	To prove \eqref{img_eq}, we will use a real Jordan basis and its corresponding Euclidean structure on $\Vc$ as in Step 2 above. 
	By the current assumption on $D$ we  obtain as above that $\exp(tD)\in\SO(\Vc)$ for every $t\in\RR$, where $\SO(\Vc)$ is the group of all isometries of $\Vc$ that can be joined by a continuous path to the identity operator~$\1$. 
	It is well known that $\SO(\Vc)$ is a compact group.  
	
	Now let $v\in\Vc$ and $x\in\omega(v)$ arbitrary. 
	Then there exists a sequence of positive real numbers $\{t_n\}_{n\ge 1}$ 
	with $\lim\limits_{n\to\infty}t_n=\infty$ and $\lim\limits_{n\to\infty}\exp(t_nD)v=x$. 
	Since $\exp(t_nD)\in\SO(\Vc)$ for all $n\ge 1$, one has $\Vert x\Vert=\Vert v\Vert<\infty$, hence $x\ne\infty$. 
	
	We choose an arbitrary sequence $\{s_n\}_{n\ge 1}$ in $(0,\infty)$ with $s_{n+1}-s_n\ge 2t_n$ for every $n\ge 1$. 
	Since $\SO(\Vc)$ is compact, there exists a sequence of positive integers $n_1<n_2<\cdots$  
	for which the sequence $\{\exp(s_{n_q} D)\}_{q\ge 1}$ is convergent. 
	In particular, for $r_q:=s_{n_{q+1}}-s_{n_q}\ge 2t_{n_q}$ we obtain 
	$$\lim_{q\to\infty}r_q=\infty\text{ and }
	\lim_{q\to\infty}\exp(r_qD)=\1.$$ 
	We also note that for all $q\ge 1$, 
	$$r_q=s_{n_{q+1}}-s_{n_q}
	= 
	\sum_{j=n_q}^{n_{q+1}-1}s_{j+1}-s_j
	\ge \sum_{j=n_q}^{n_{q+1}-1} 2t_j
	\ge 2t_{n_q}.$$
	For $a_q:=r_q-t_{n_q}\ge t_{n_q}$ we have $\lim\limits_{q\to\infty}a_q=\infty$ and, 
	by $\exp(-a_q D)\in\SO(\Vc)$, 
	$$\begin{aligned}
	\Vert\exp(-a_qD)v-x\Vert
	& \le \Vert\exp(-a_qD)v-\exp(t_{n_q} D)v\Vert
	+\Vert\exp(t_{n_q}D)v-x\Vert \\
	&=\Vert v-\exp((a_q+t_{n_q})D)v\Vert
	+\Vert\exp(t_{n_q}D)v-x\Vert \\
	&=\Vert v-\exp(r_qD)v\Vert
	+\Vert\exp(t_{n_q}D)v-x\Vert\to 0\text{ as }q\to\infty
	\end{aligned}$$
	hence $x\in\omega^*(v)$. 
	This shows that $\omega(v)\subseteq\omega^*(v)$. 
	The converse inclusion can be proved similarly, 
	hence \eqref{img_eq} is completely proved, and we are done. 
\end{proof}

\begin{proposition}\label{attr}
	Let $\Vc$ be a finite-dimensional real vector space. 
	If $D\in \End(\Vc)$ then its corresponding flow 
	$$\alpha_D\colon \overline{\Vc}\times\RR\to\overline{\Vc}, \quad \alpha_D(v,t):=\exp(tD)v$$ 
	has a nontrivial  attractor-repeller pair if and only if either $\Re z>0$ for every $z\in\sigma(D)$ 
	or $\Re z<0$ for every $z\in\sigma(D)$. 
	If this is the case, then the only nontrivial  attractor-repeller pair is $(\{\infty\},\{0\})$ or $(\{0\},\{\infty\})$, 
	respectively. 
\end{proposition}

\begin{proof}
	If $\Re z<0$ for every $z\in\sigma(D)$, then $0>\lambda_1>\cdots>\lambda_\ell$, 
	hence by \eqref{attr_proof_eq1} we easily obtain that for arbitrary $v\in\Vc$ 
	there exists $t_0>0$ such that 
	if $t\in(t_0,\infty)$, then $\frac{1}{t}\log\Vert\exp(tD)v\Vert<\lambda_1/2$, 
	that is, $\Vert\exp(tD)v\Vert<\exp(t\lambda_1/2)$, 
	and then 
	$\lim\limits_{t\to\infty}\exp(tD)v=0$. 
	This shows that $\omega(v)=\{0\}$ for every $v\in\Vc$, which directly implies 
	that there exists a unique nontrivial  attractor-repeller pair, namely $(\{0\},\{\infty\})$. 
	
	In the case when $\Re z>0$ for every $z\in\sigma(D)$, 
	we obtain just as above, using however \eqref{attr_proof_eq2} instead of \eqref{attr_proof_eq1}, 
	that for every $v\in\Vc\setminus\{0\}$ one has $\lim\limits_{t\to\infty}\exp(tD)v=\infty$, 
	hence $\omega(v)=\{\infty\}\subset\overline{\Vc}$, 
	and this implies that there exists a unique nontrivial attractor-repeller pair, namely $(\{\infty\},\{0\})$. 
	
	By Lemma~\ref{img}, it remains to show that if $\lambda_1> 0 >\lambda_\ell$, then 
	there exists no nontrivial attractor-repeller pair. 
	Reasoning by contradiction, let us assume that there exists an attractor $A$ with $Z:=\overline{\Vc}\setminus(A\sqcup A^*)\ne\emptyset$. 

The inequalities $\lambda_1> 0>\lambda_\ell$ imply in particular $\ell\ge 2$, 
	and then $\Vc_2\subsetneqq\Vc$ and $\Wc_{\ell-1}\subsetneqq\Vc$. 
	
	We note that $Z=\overline{\Vc}\setminus(A\sqcup A^*)$ is an open subset of $\overline{\Vc}$, 
	and $Z\ne\emptyset$ by assumption. 
	As noted at the beginning of the proof of Lemma~\ref{img}, one has $Z\subseteq\Vc\setminus \{0\}$. 
		Then $Z$ is an open nonempty subset of $\Vc$, hence one has neither $Z\subseteq\Vc_2$ nor 
	$Z\subseteq\Wc_{\ell-1}$. 
	That is, there exist $z_1\in Z\setminus\Vc_2$ and 
	$z_2\in Z \setminus\Wc_{\ell-1}$. 
	Then, using \eqref{attr_proof_eq5}, we obtain 
	$\{\infty\}\subseteq\omega(z_1)\subseteq\omega(Z)\subseteq A$, 
	and on the other hand, by \eqref{attr_proof_eq6} we obtain 
	$\{\infty\}\subseteq\omega^*(z_2)\subseteq\omega^*(Z)\subseteq A^*$, 
	hence $\infty\in A\cap A^*$, which is a contradiction with the fact that always $A\cap A^*=\emptyset$. 
	This completes the proof.
	 \end{proof}

\begin{theorem}\label{ax+b}
	Let $\Vc$ be a finite-dimensional real vector space,   
	fix $D\in \End(\Vc)$. 
	The following assertions are equivalent: 
	\begin{enumerate}[{\rm(i)}]
		\item\label{ax+b_spec} One has either $\Re z>0$ for every $z\in\sigma(D)$ 
		or $\Re z<0$ for every $z\in\sigma(D)$.  
		\item\label{ax+b_qsd} The $C^*$-algebra $C^*(G_D)$ is not quasidiagonal. 
		\item\label{ax+b_emb} There exists no embedding of $C^*(G_D)$ into any AF-algebra. 
		 \end{enumerate}
\end{theorem}

\begin{proof} 
	We denote $G:=G_D$ for simplicity. 
	
	The implication $
	\eqref{ax+b_qsd}\implies\eqref{ax+b_emb}$ holds true with $C^*(G)$ replaced by any $C^*$-algebra. 
	
	For $\eqref{ax+b_spec}\iff\eqref{ax+b_qsd}$ we note that Assertion~\eqref{ax+b_spec} is equivalent to the similar property of the linear map $D^*\in\End(\Vc^*)$. 
	By Proposition~\ref{attr}, the Assertion~\eqref{ax+b_spec} is equivalent to existence of a nontrivial attractor-repeller pair 
	for the flow $\alpha_{D^*}\colon \overline{\Vc^*}\times\RR\to\overline{\Vc^*}$, 
	and by Lemma~\ref{comp} this is further equivalent 
	to any of the following assertions: 
	\begin{enumerate}[(a)]
		\item\label{ax+b_proof_item_a} The $C^*$-algebra $\RR\ltimes_{\alpha_{D^*}}\Mc$ fails to be quasidiagonal. 
		\item\label{ax+b_proof_item_b} There exists no embedding of  
		$\RR\ltimes_{\alpha_{D^*}}\Mc$ into any AF-algebra. 
	\end{enumerate}
	Here we denoted  $\Mc:=\CC\1+\Cc_0(\Vc^*)\simeq\Cc(\overline{\Vc^*})$.  
	Recall from \cite[Ex. 3.16]{Wi07} that there is a $*$-isomorphism 
	\begin{equation}
	\label{ax+b_proof_eq1}
	C^*(G)\simeq\RR\ltimes_{\alpha_{D^*}}\Cc_0(\Vc^*). 
	\end{equation}
	Then, by Lemma~\ref{crossunit} and 
	\eqref{ax+b_proof_eq1}, the above Assertion~\eqref{ax+b_proof_item_a} is equivalent to the fact that 
	$C^*(G)$ fails to be quasidiagonal, 
	and this completes the proof of $\eqref{ax+b_spec}\iff\eqref{ax+b_qsd}$.  
	
	For $\eqref{ax+b_emb}\implies\eqref{ax+b_spec}$ we note that if 
	\eqref{ax+b_spec} does not hold true, then 
	the above assertion~\eqref{ax+b_proof_item_b} also fails to be true 
	(by the above discussion), hence there exists an embedding of $\RR\ltimes_{\alpha_{D^*}}\Mc$ into some AF-algebra. 
	In particular, the ideal 	$C^*(G)\simeq\RR\ltimes_{\alpha_{D^*}}\Cc_0(\Vc^*)$ of 
	$\RR\ltimes_{\alpha_{D^*}}\Mc$ embeds into some AF-algebra, hence \eqref{ax+b_emb} fails to be true, and this completes the proof. 
\end{proof}

We now draw a corollary of the above theorem that also needs Theorem~\ref{main1}, proved in the next section. 
We point out however that the proof of Theorem~\ref{main1} 
relies only on Theorem~\ref{ax+b} and not on this corollary. 
We decided to give this result here in order to draw a more complete picture of the quasidiagonality properties of the generalized $ax+b$-groups. 

\begin{corollary}\label{ax+b_cor1}
	Let $\Vc$ be a finite-dimensional real vector space 
	and select any $D\in\End(\Vc)$ whose spectrum does not contain any nonzero purely imaginary eigenvalues. 
	Then $G_D$ is an exponential solvable Lie group, and moreover $C^*(G_D)$ is quasidiagonal but not strongly quasidiagonal if and only if $D$ satisfies the following spectral condition: 
	\begin{itemize}
		\item One has $\sigma(D)\ne\{0\}$ and there exist $z_1,z_2\in\sigma(D)$ with $\Re z_1\le 0\le \Re z_2$. 
	\end{itemize}
\end{corollary}

\begin{proof}
	Existence of $z_1,z_2\in\sigma(D)$ with $\Re z_1\le 0\le \Re z_2$ is equivalent to the fact that Assertion~\eqref{ax+b_spec} in Theorem~\ref{ax+b} does not hold true, and this is equivalent to the fact that $C^*(G_D)$ is quasidiagonal. 	
	
	Moreover, the spectrum of  $D\in\End(\Vc)$ does not contain nonzero purely imaginary eigenvalues if and only if the solvable Lie group $G_D$ is exponential (see e.g., \cite{FuLu15}). 
	Then, by Theorem~\ref{main1}, $C^*(G_D)$ is not strongly quasidiagonal if and only if the Lie group $G_D$ is not nilpotent, 
	which is further equivalent to the fact that 
	that the map~$D$ is not nilpotent, that is, $\sigma(D)\ne\{0\}$. 
\end{proof}

\begin{corollary}\label{ax+b_cor2}
	Let $\Vc$ be a real vector space with $m:=\dim\Vc<\infty$, 
	and denote by $\End_0(\Vc)$ the set of all 
	endomorphisms $D\in\End(\Vc)$ with $\Vc^D_0=\Ker D$. 
	
	There exist exactly 
	$$N(m):=\sum\limits_{n_0=0}^{m}\Bigl(1+\Bigl[\frac{m-n_0}{2}\Bigr]\Bigr)\ge m+1$$
	$*$-isomorphism classes in the set of $C^*$-algebras $\{C^*(G_D)\mid D\in\End_0(\Vc)\}$, and exactly one of these $*$-isomorphism classes contains non-quasi\-diagonal $C^*$-algebras.
\end{corollary}

\begin{proof}
For every $D\in\End_0(\Vc)$ we denote $n_0^D:=\dim(\Ker D)$. 
 	It follows by Theorem~\ref{isom-dir} and Theorem~\ref{isom-conv} 
	that for $D_1,D_2\in\End_0(\Vc)$ one has a $*$-isomorphism $C^*(G_{D_1})\simeq C^*(G_{D_2})$ if and only if one has both  $n_0^{D_1}=n_0^{D_2}$ and 
	$\{n^{D_1}_\epsilon\mid\epsilon=\pm\}=\{n^{D_2}_\epsilon\mid\epsilon=\pm\}$. 
	Therefore, it easily follows that, with $N(m)$ defined in the statement, 
	there exist exactly 
	$N(m)$
	$*$-isomorphism classes in the set of $C^*$-algebras $\{C^*(G_D)\mid D\in\End_0(\Vc)\}$. 
	By Theorem~\ref{ax+b}, exactly one of these $*$-isomorphism classes corresponds to non-quasi\-diagonal $C^*$-algebras, 
	namely for $D\in\End_0(\Vc)$ with $n_0^D=n^D_{+}n^D_{-}=0$. 
\end{proof}

As an example of the above Theorem~\ref{ax+b},  we obtain that the $C^*$-algebra of the Mautner group $\RR^4 \rtimes_{\alpha_D} \RR$ is quasidiagonal, 
where 
$$D=\begin{pmatrix} \ie & 0 \\
                        0 & \ie\theta 
                        \end{pmatrix}\in M_2(\CC)\simeq \End(\RR^4)$$
with $\theta\in\RR\setminus\QQ$.                          
It is well known that it is a solvable Lie group that is not of type~I 
(see for instance \cite{AuMo66} or \cite[Prop. 2.1]{AbArSe04}).

Existence of solvable Lie groups whose $C^*$-algebras are not quasidiagonal, which follows by Corollaries \ref{ax+b_cor1}--\ref{ax+b_cor2}, shows that the analogue of Rosenberg's conjecture 
(see \cite[Cor. C]{TiWhWi16}) does not carry over directly to non-discrete groups.

\section{Proof of Theorem~\ref{main1}}
\label{sect3} 

We first establish the most direct implications between the properties 
from the statement. 
	
	$\eqref{main1_CCR}\implies\eqref{main1_strqsd}$: 
	This is well known.   
	See for instance \cite[Prop. 8(2)]{Hd87}. 
	
	$\eqref{main1_ad}\iff\eqref{main1_CCR}$: 
	This follows by \cite[Ch. V, Thms. 1--2]{AuMo66}. 
	
	
	For any exponential solvable Lie group $G$, one has  
	$\eqref{main1_CCR}\iff\eqref{main1_nilp}$ 
	by \cite[Rem. 2.11]{BBG16}. 
	Finally, $\eqref{main1_nilp}\iff\eqref{main1_nilp2}$ by \cite[Th. 5.3.31]{FuLu15}.

We now begin the preparations for 
the remaining step in the proof of Theorem~\ref{main1}, namely  $\eqref{main1_strqsd}\implies\eqref{main1_CCR}$. 

\begin{proposition}\label{comp1}
	Let $G$ be a locally compact group with a $C^*$-dynamical system $G\times \Ac\to\Ac$, $(g,a)\mapsto g\cdot a$ and a normal subgroup~$K\subseteq G$ 
	for which we define the fixed-point subalgebra $\Ac^K:=\{a\in\Ac\mid (\forall k\in K)\ k\cdot a=a\}$.   
	Then the following assertions hold: 
	\begin{enumerate}[{\rm(i)}]
		\item\label{comp1_item1} $\Ac^K$ is a $G$-invariant closed $*$-subalgebra of $\Ac$. 
		\item\label{comp1_item2} If $K$ is additionally assumed to be a compact subgroup of $G$, 
		$S$ is another amenable closed subgroup of $G$ with $G=SK$ and $S\cap K=\{\1\}$, and, the probability Haar measure of $K$ is invariant under the map $k\mapsto sks^{-1}$ for every $s\in S$, then there exists an injective $*$-morphism $S\ltimes \Ac^K\hookrightarrow G\ltimes\Ac$.
	\end{enumerate}
\end{proposition}

\begin{proof}
	For every $a\in\Ac$, $g\in G$ and $k\in K$ one has 
	$$k\cdot(g\cdot a)=g\cdot((g^{-1}kg)\cdot a)$$
	where $g^{-1}kg\in K$ since $K$ is a normal subgroup of $G$. 
	Therefore, if $a\in \Ac^K$ then $g\cdot a\in\Ac^K$ for arbitrary $g\in G$. 
	
	We will now prove Assertion~\eqref{comp1_item2}. 
	Since $K$ is a compact group, one has the injective $*$-morphism 
	$$\eta\colon \Ac^K\to L^1(K,\Ac)\hookrightarrow K\ltimes\Ac,
	\quad (\eta(a))(k)=a 
	\text{ for all }a\in\Ac^K,\, k\in K.$$
	(See \cite{Ro79}.) 
	
	We now make explicit certain natural actions of $S$ on $\Ac^K$ and on $K\ltimes\Ac$, and check that $\eta$ is $S$-equivariant with respect to these actions. 
	It follows by Assertion~\eqref{comp1_item1} that the action of $G$ on $\Ac$ induces an action of $G$ on $\Ac^K$, which further gives by restriction an action of $S$ on $\Ac^K$. 
	To describe the action of $S$ on $K\ltimes\Ac$, we first note the semidirect product decomposition $G=S\ltimes K$, which follows by the hypothesis. 
	We then obtain by \cite[Prop. 3.11]{Wi07} a $C^*$-dynamical system 
	$S\times (K\ltimes\Ac)\to K\ltimes \Ac$ given by  
	$$(s\cdot f)(k)=s\cdot (f(s^{-1}ks))\text{ for all }s\in S\text{ and } f\in L^1(K,\Ac)\hookrightarrow K\ltimes\Ac,$$ 
	where we have used the fact that $K$ is a normal subgroup of $G$ and the action of $S$ on $\Ac$ given by the $C^*$-dynamical system $G\times \Ac\to\Ac$. 
	With respect to these actions of $S$ on $\Ac^K$ and on $K\ltimes\Ac$, 
	one clearly has 
	$$(\forall s\in S)(\forall a\in\Ac^K)\quad 
	\eta(s\cdot a)=s\cdot(\eta(a)).$$
	Since $S$ is an amenable group, crossed products by $L$ are canonically isomorphic to their corresponding reduced crossed products, hence by the above $S$-equivariance property of the  injective $*$-morphism $\eta$  obtain an injective $*$-morphism 
	$$\widetilde{\eta}\colon S\ltimes\Ac^K\to S\ltimes(K\ltimes\Ac)$$
	for instance by \cite[Prop. 7.7.9]{Pe79}. 
	Using the $*$-isomorphism $S\ltimes(K\ltimes\Ac)\simeq (S\ltimes K)\ltimes\Ac$ given by \cite[Prop. 3.11]{Wi07}, 
	the above map $\widetilde{\eta}$ gives an injective $*$-morphism $S\ltimes \Ac^K\hookrightarrow G\ltimes\Ac$, we are done. 
\end{proof}

\begin{corollary}\label{comp2}
	Let $G$ be a locally compact group two closed subgroups $K$ and $S$ satisfying $G=SK$, $S\cap K=\{\1\}$, and $sk=ks$ for all $s\in S$ and $k\in K$. 
	We also assume that $K$ is compact and we fix 
	a $C^*$-dynamical system $G\times \Ac\to\Ac$, $(g,a)\mapsto g\cdot a$. 
	
	Then $\Ac^K:=\{a\in\Ac\mid (\forall k\in K)\ k\cdot a=a\}$ is an $S$-invariant closed $*$-subalgebra of~$\Ac$ whose corresponding crossed product $S\ltimes\Ac^K$ is a quasidiagonal $C^*$-algebra (respectively, is AF-embeddable) if $G\ltimes\Ac$ is. 
\end{corollary}

\begin{proof}
	It follows directly from Proposition~\ref{comp1}\eqref{comp1_item2} 
	that $S\ltimes\Ac^K$ is $*$-isomorphic to a closed $*$-subalgebra 
	of $G\ltimes\Ac$.
\end{proof}

\begin{example}\label{eucl}
	\normalfont
	For every integer $n\ge 1$, we denote by $SO(n)$ its corresponding special orthogonal group, that is, the group of all orthogonal matrices $T\in M_n(\RR)$ with $\det{T}=1$. It is well known that
	$SO(n)$ is a connected compact Lie group.

	Since the tautological action of $K:=\SO(n)$ on $\RR^n$ is an action by linear maps, 
	it commutes with the action of the multiplicative group $(0,\infty)$ on $\RR^n$ 
	by $(r,x)\mapsto rx$. 
	One has the group isomorphism  $S:=(\RR,+)\simeq((0,\infty),\cdot)$, $s\mapsto\ee^s$, hence one thus obtains an action 
	of the group $G:=K\times S$ on $\RR^n$, 
	\begin{equation}
	\label{eucl_eq0}
	\RR^n\times G \to\RR^n,\quad 
	(x, (T,s))\mapsto T(e^sx)=\ee^s(Tx)
	\end{equation}
	for $(T,s)\in K\times S=G$ and $x\in\RR^n$.
	This defines a $C^*$-dynamical system $G\times\Ac\to\Ac$ 
	for the commutative $C^*$-algebra $\Ac:=\Cc_0(\RR^n)$. 
	
	In the notation of Corollary~\ref{comp2} one clearly has the $*$-isomorphism 
	$$\Cc_0([0,\infty))\to\Ac^K,\quad f\mapsto
	f(\vert\cdot\vert),$$
	where $\vert\cdot\vert$ is the Euclidean norm on $\RR^n$. 
	One can transport the action of $S$ from $\Ac^K$ to $\Cc_0([0,\infty))$ via this $*$-isomorphism, and one thus obtains the $C^*$-dynamical system 
	$$S\times \Cc_0([0,\infty))\to \Cc_0([0,\infty)),\quad 
	(s,h)\mapsto s\cdot \varphi,$$
	where $(s\cdot \varphi)(r):=\varphi(\ee^s r)$ for all $s\in \RR$, $r\in[0,\infty)$, and $\varphi\in\Cc_0([0,\infty))$. 
	
	On the other hand, the unitization of the commutative $C^*$-algebra $\Cc_0([0,\infty))$ is $*$-isomorphic to 
	$\Cc([0,\infty])$, and one has a canonical extension 
	of the above action of $S$ on $\Cc_0([0,\infty))$ to 
	an action of $S$ on $\Cc([0,\infty])$. 
	The corresponding crossed product $S\ltimes \Cc([0,\infty])$ is the $C^*$-algebra of the flow 
	$$S\times [0,\infty]\to[0,\infty],\quad (s,r)\mapsto \ee^s r.$$
	It is clear that the pair $(\{\infty\},\{0\})$ is a nontrivial attractor-repeller pair for this flow, hence it follows by
	Proposition~\ref{comp} that the $C^*$-algebra  $S\ltimes \Cc_0([0,\infty))$ is not quasidiagonal. 
	Therefore $S\ltimes\Ac^K$ is not quasidiagonal, and then 
	by Corollary~\ref{comp2} we obtain that the $C^*$-algebra 
	$G\ltimes\Ac$, that is, $(\SO(n)\times \RR)\ltimes\Cc_0(\RR^n)$, 
	is not quasidiagonal. 
\end{example}

\begin{example}\label{eucl_rem}
	\normalfont
	In connection with Example~\ref{eucl}, we recall that the Euclidean motion group $E(n):=\RR^n \rtimes \SO(n)$ is the semidirect product defined via the tautological action of 	$K:=\SO(n)$ on $\RR^n$, hence in particular 
	$$C^*(E(n))\simeq\SO(n)\ltimes\Cc_0(\RR^n).$$
	This is a liminal algebra for instance by \cite[Th.~3.1]{Wi81}, therefore it is strongly quasidiagonal. 
	The $C^*$-algebras of such groups were recently studied in \cite{AraH15}. 
	
	On the other hand, there is also the action of $S=(\RR,+)$ on $\RR^n$ as in Example~\ref{eucl}, which defines the following action of $\RR$ by automorphisms of $E(n)$, 
	$$\RR\times E(n)\to E(n),\quad (s,(x,T ))\mapsto (\ee^s x, T )$$
	and which further defines a natural action of $\RR$ on $C^*(E(n))$. 
	We will show that the corresponding crossed product $\RR\ltimes C^*(E(n))$ is not quasidiagonal. 
	(We recall however from \cite[Cor.~11.2]{Br04} that the crossed product of any quasidiagonal $C^*$-algebra by a separable compact group remains quasidiagonal.)
	
	To this end we first note the Lie group isomorphism $$ E(n)\rtimes\RR \to \RR^n \rtimes  (\SO(n)\times\RR),\quad 
	((x,T),  s)\mapsto (x, (T,s))$$ 
	which defines a $*$-isomorphism $C^*(E(n) \rtimes \RR)\simeq C^*(\RR^n\rtimes (\SO(n)\times\RR))$. 
	Since $C^*(E(n) \rtimes \RR )\simeq \RR\ltimes C^*(E(n))$ 
	and $C^*(\RR^n \rtimes (\SO(n)\times\RR))\simeq (\SO(n)\times\RR)\ltimes\Cc_0(\RR^n)$, 
	it then follows by Example~\ref{eucl} that 
	$\RR\ltimes C^*(E(n))$  is not quasidiagonal. 
\end{example}

We now prove that if \eqref{main1_CCR} in Theorem~\ref{main1} fails to be true, 
then \eqref{main1_strqsd} is not true. 
To this end we show that there exists a closed 2-sided ideal $\Jc\subseteq C^*(G)$ for which the quotient $C^*(G)/\Jc$ is not strongly quasidiagonal, 
and then $C^*(G)$ cannot be strongly quasidiagonal. 
The proof of this fact is divided in three steps. 
The first two show how  to reduce the problem to certain low-dimension groups, while the third step treats the case of  those low-dimension groups.

\subsection*{Step 1}
It follows by \cite[Th. 1--2 and Prop. 2.2, Ch. V]{AuMo66} 
that if $G$ is a connected simply connected solvable Lie group 
of type I, then $G$ fails to be a liminal group 
if and only if there exists a connected simply connected closed normal subgroup $N\subseteq G$ 
for which the quotient Lie group $G/N$ is isomorphic to one of the following Lie groups: 
\begin{enumerate}
	\item $S_2:=\RR\rtimes\RR$, the connected real $ax+b$-group, defined via 
	$$\alpha\colon (\RR,+)\to\Aut(\RR,+), \quad \alpha(t)s=\ee^t s$$ 
	\item $S_3^\sigma:=\RR^2\rtimes_{\alpha^\sigma}\RR$, defined for $\sigma\in\RR\setminus\{0\}$ via  
	$$\alpha^\sigma\colon (\RR,+)\to\Aut(\RR^2,+), \quad 
	\alpha^\sigma(t)=\ee^{\sigma t}
	\begin{pmatrix}
	\hfill \cos t & \sin t \\
	-\sin t & \cos t
	\end{pmatrix}$$
	\item $S_4:=\RR^2\rtimes_{\beta}\RR^2$, defined via  
	$$\beta\colon (\RR^2,+)\to\Aut(\RR^2,+), \quad 
	\beta(t,s)=\ee^t
	\begin{pmatrix}
	\hfill \cos s & \sin s \\
	-\sin s & \cos s
	\end{pmatrix}$$
\end{enumerate}

\subsection*{Step 2}
Every short exact sequence of amenable locally compact groups 
$$\1\to N\to G\to G/N\to\1$$
leads to a short exact sequence of $C^*$-algebras 
$$0\to\Jc\to C^*(G)\to C^*(G/N)\to 0$$
for a suitable closed 2-sided ideal $\Jc$ of $C^*(G)$. 
This shows that in order to implement the method of proof described above,  
it suffices to check that the $C^*$-algebra of the above groups 
$S_2$, $S_3^\sigma$ with $\sigma\in\RR\setminus\{0\}$, and $S_4$ are not strongly quasidiagonal. 

\subsection*{Step 3}
Using Theorem~\ref{ax+b} it follows directly that the $C^*$-algebra of any of the groups 
$S_2$ and $S_3^\sigma$ with $\sigma\in\RR\setminus\{0\}$ is not strongly quasidiagonal.

We now discuss the case of the group $S_4$. 
It is easily seen that the surjective map  
$$S_4=\RR^2\ltimes_\beta\RR\to \RR^2\rtimes_\alpha (\RR\times\SO(2)),\quad 
(x,(t,s))\mapsto (x,(t,\begin{pmatrix}
\hfill \cos s & \sin s \\
-\sin s & \cos s
\end{pmatrix}))$$
is a group homomorphism  (whose kernel is isomorphic to $\ZZ$), 
where $\alpha$ is the group action from \eqref{eucl_eq0} for $n=2$. 
As in Step~2 above, one then obtains a short exact sequence of $C^*$-algebras 
$$0\to\Ic\to C^*(S_4)\to C^*(\RR^2\rtimes_\alpha (\RR\times\SO(2)))\to 0$$
for a suitable closed 2-sided ideal $\Ic$ of $C^*(S_4)$. 
The special case $n=2$ of Example~\ref{eucl} shows that the 
$C^*$-algebra $C^*(\RR^2\rtimes_\alpha (\RR\times\SO(2)))$ is not quasidiagonal, 
and then the above short exact sequence shows that $C^*(S_4)$ cannot be strongly quasidiagonal. 

This completes the proof of Theorem~\ref{main1}.

\section{On faithful tracial states of $C^*$-algebras of solvable groups}
\label{sect4} 

In this final section we show that the $C^*$-algebras of the groups we have been studied in the previous sections do not admit faithful tracial states. Thus we are in a situation that is complementary to the situation in \cite{TiWhWi16}:
The reduced  $C^*$-algebra of any discrete group admits a  canonical faithful tracial state which is quasidiagonal when the group is countable and amenable, which in turn implies the quasidiagonality of the reduced $C^*$-algebra of the group.

\begin{lemma}\label{SophusLie}
	If $G$ is a connected topological solvable group and $\pi\colon G\to\Bc(\Hc)$ is a continuous irreducible representation with $\dim\Hc<\infty$, then $\dim\Hc\le 1$.
\end{lemma}

\begin{proof}
	See \cite[Ch. 8, \S 1, Th. 1]{BaRa86}.
\end{proof}

The following Lemma is well known, but we insert its proof for the sake of completness. 

\begin{lemma}\label{faith-tr}
	Let $\Ac$ be a $C^*$-algebra with a faithful state $\varphi\colon\Ac\to\CC$. 
	Assume that  $\pi_\varphi\colon \Ac\to\Bc(\Hc_\varphi)$ is a $*$-representation with a cyclic vector $x_\varphi\in\Hc_\varphi$ with $\Vert x_\varphi\Vert=1$ and $\varphi(a)=(\pi_\varphi(a)x_\varphi\mid x_\varphi)$ for every $a\in\Ac$. 
	Then
	the following assertions hold: 
	\begin{enumerate}[{\rm(i)}]
		\item One has $\Ker\pi_\varphi=\{0\}$. 
		\item If the faithful state $\varphi$ is moreover a tracial state, then the functional 
		$$\tau_\varphi\colon\pi_\varphi(\Ac)''\to\CC,\quad   \tau_\varphi(T):=(Tx_\varphi\mid x_\varphi)$$
		is a normal faithful tracial state on the von Neumann algebra~$\pi_\varphi(\Ac)''$. 
	\end{enumerate}
\end{lemma}

\begin{proof}
	For the first assertion, if $\pi_\varphi(a)=0$ then 
	$(\pi_\varphi(a^*a)x_\varphi\mid x_\varphi)=\varphi(a^*a)=0$. 
	Since the state $\varphi$ is faithful, we then obtain $a^*a=0$, hence $a=0$. 
	
	For the second assertion, let us assume that the faithful state $\varphi$ is tracial. 
	The functional $\tau_\varphi$ in the statement is clearly a state of $\pi_\varphi(\Ac)''$ and is continuous with respect to the weak operator topology, hence it is in particular a normal state of $\pi_\varphi(\Ac)''$. 
	
	In order to check that $\tau_\varphi$ is a tracial state, 
	we first note that $\tau_\varphi(\pi_\varphi(a))=\varphi(a)$ for every $a\in\Ac$. 
	Therefore, if $T_j=\pi_\varphi(a_j)$ for $j=1,2$, then, 
	using the fact that $\varphi$ is a tracial state, we obtain 
	$$\tau_\varphi(T_1T_2)=\tau_\varphi(\pi_\varphi(a_1a_2))
	=\varphi(a_1a_2)=\varphi(a_2a_1)\tau_\varphi(\pi_\varphi(a_2a_1))
	=\tau_\varphi(T_2T_1).$$
	Since we have already seen that $\tau_\varphi$ is a normal state 
	and on the other hand the bicommutant theorem ensures that $\pi_\varphi(\Ac)$ is dense in $\pi_\varphi(\Ac)''$ in the strong operator topology, the equality 
	$\tau_\varphi(T_1T_2)=\tau_\varphi(T_2T_1)$ extends to all 
	$T_1,T_2\in\pi_\varphi(\Ac)''$. 
	
	It remains to prove that $\tau_\varphi$ is faithful. 
	To this end, let $T\in\pi_\varphi(\Ac)''$ 
	with $\tau_\varphi(T^*T)=0$, which is equivalent to $Tx_\varphi=0$. 
	Then for every $a\in\Ac$ one has 
	$$\Vert T\pi_\varphi(a)x_\varphi\Vert^2
	=(\pi_\varphi(a)^*T^*T\pi_\varphi(a)x_\varphi \mid x_\varphi)
	=\tau_\varphi(\pi_\varphi(a)^*T^*T\pi_\varphi(a)).$$
	We have already seen that $\tau_\varphi$ is a tracial state of $\pi_\varphi(\Ac)''$, hence we further obtain 
	$$\Vert T\pi_\varphi(a)x_\varphi\Vert^2
	=\tau_\varphi(\pi_\varphi(a)\pi_\varphi(a)^*T^*T)
	=(\pi_\varphi(a)\pi_\varphi(a)^*T^*T x_\varphi \mid x_\varphi)=0$$
	where we used $Tx_\varphi=0$. 
	Thus $T\pi_\varphi(a)x_\varphi=0$ for every $a\in\Ac$. 
	Since $x_\varphi$ is a cyclic vector for the representation~$\pi_\varphi$, it follows that $T=0$, hence the state $\tau_\varphi$ is faithful, and we are done. 
\end{proof}

\begin{lemma}\label{RFD}
	If $\Ac$ is a $C^*$-algebra of type~I which has a  
	faithful tracial state, then the following assertions hold:  
	\begin{enumerate}[{\rm(i)}]
		\item If $\{\pi_j\}_{j\in J}$ is a complete system of distinct representatives of the unitary equivalence classes of finite-dimensional irreducible $*$-representations of $\Ac$, 
		then the $*$-representation $\bigoplus_{j\in J}\pi_j$ 
		is faithful. 
		\item If for every irreducible $*$-representation $\pi\colon\Ac\to\Bc(\Hc)$ one has $\dim\Hc\le 1$, then $\Ac$ is commutative. 
	\end{enumerate}
\end{lemma}

\begin{proof}
	The first assertion is a more detailed statement of \cite[Prop. 7.1.8]{BrOz08}, whose proof also requires the above Lemma~\ref{faith-tr}. 
	The second assertion follows directly from the first assertion. 
\end{proof}

\begin{proposition}
	Let $G$ be a connected locally compact solvable group of type~I. 
	If $C^*(G)$ has a faithful tracial state, then $G$ is commutative. 
\end{proposition}

\begin{proof}
	It is well known that the group $G$ is commutative if and only if the Banach algebra $L^1(G)$ is commutative. 
	This is further equivalent to the property that $C^*(G)$ be commutative, since $L^1(G)$ is dense in $C^*(G)$. 
	
	On the other hand, $C^*(G)$ is a $C^*$-algebra of type~I with a faithful tracial state, 
	and for every irreducible $*$-representation of $\pi\colon C^*(G)\to\Bc(\Hc)$ one has $\dim\Hc\le1$ as a direct consequence of Lemma~\ref{SophusLie}.  
	Therefore $C^*(G)$ is commutative by Lemma~\ref{RFD}, and this concludes the proof. 
\end{proof}


\end{document}